\documentclass[11pt]{article}

\usepackage{amsfonts}
\usepackage{shuffle,yfonts,amsmath,amssymb,amsthm,etex,tikz}

\newcommand\Res{{\rm Res}}

\newcommand{\bfk}{{\boldsymbol{\sl{k}}}}

\newcommand\bfeta{{\boldsymbol \eta}}

 \allowdisplaybreaks
 \allowdisplaybreaks

\catcode`!=11
\let\!int\int \def\int{\displaystyle\!int}
\let\!lim\lim \def\lim{\displaystyle\!lim}
\let\!sum\sum \def\sum{\displaystyle\!sum}
\let\!sup\sup \def\sup{\displaystyle\!sup}
\let\!inf\inf \def\inf{\displaystyle\!inf}
\let\!cap\cap \def\cap{\displaystyle\!cap}
\let\!max\max \def\max{\displaystyle\!max}
\let\!min\min \def\min{\displaystyle\!min}
\let\!frac\frac \def\frac{\displaystyle\!frac}
\catcode`!=12

\let\oldsection\section
\renewcommand\section{\setcounter{equation}{0}\oldsection}

\allowdisplaybreaks

\DeclareMathOperator*{\dep}{dep}

\def\N{\mathbb{N}}
\def\Z{\mathbb{Z}}
\def\Q{\mathbb{Q}}

\def\ze{\zeta}

\theoremstyle{plain}
\newtheorem{thm}{Theorem}[section]
\newtheorem{lem}[thm]{Lemma}
\newtheorem{cor}[thm]{Corollary}

\theoremstyle{definition}

\newtheorem{re}[thm]{Remark}

\begin{document}
\title{\bf Parametric Euler Sums of Harmonic Numbers}
\author{
{Junjie Quan${}^{a,}$\thanks{Email:  as6836039@163.com} \quad Xiyu Wang${}^{b,}$ \thanks{Email: xiyuwang2021@outlook.com.} \quad Xiaoxue Wei${}^{c,}$\thanks{Email: xiaoxueweidufe@163.com.}\quad Ce Xu${}^{b,}$\thanks{Email: cexu2020@ahnu.edu.cn}}\\[1mm]
\small a. School of Information Science and Technology,  Xiamen University Tan Kah Kee College\\
\small Xiamen Fujian 363105, P.R. China\\
\small b. School of Mathematics and Statistics, Anhui Normal University,\\ \small Wuhu 241002, P.R. China\\
\small c. School of Economics and Management, Anhui Normal University,\\ \small Wuhu 241002, P.R. China
}

\date{}
\maketitle

\noindent{\bf Abstract.} We define a parametric variant of generalized Euler sums and construct contour integration to give some explicit evaluations of these parametric Euler sums. In particular, we establish several explicit formulas of (Hurwitz) zeta functions, linear and quadratic parametric Euler sums. Furthermore, we also give an explicit evaluation of alternating double zeta values $\ze(\overline{2j},2m+1)$ in terms of a combination of alternating Riemann zeta values by using the parametric Euler sums.


\medskip
\noindent{\bf Declaration of interest.} None.

\section{Introduction}
We begin with some basic notations. Let $\Z,\N$ and $\N^-$ be the set of integers, positive integers and negative integers, respectively, $\N_0:=\N\cup \{0\}$ and $\N^-_0:=\N^-\cup \{0\}$. For $p\in\N$ and $n\in\N_0$, the \emph{generalized harmonic number} $H_n^{(p)}$ is defined by
\begin{align}
H_n^{(p)}:=\sum_{k=1}^n \frac{1}{k^p}\quad \text{and}\quad H_0^{(p)}:=0,
\end{align}
where if $p=1$ then $H_n\equiv H_n^{(1)}$ is the \emph{classical harmonic number}.

Between late 1742 and early 1743, Euler first touched on the subject of the \emph{linear Euler sums} in a series of correspondence with Goldbach. In modern notation, these are defined
as follows:
\begin{align}
S_{p;q}:=\sum_{n\geq k\geq 1}^\infty \frac{1}{k^pn^q}=\sum_{n=1}^\infty \frac{H_n^{(p)}}{n^q},
\end{align}
where $p,q\in \N$ and $q\geq 2$ is to ensure convergence of the series.

Euler returned to the same subject after about 30 years and discovered the now famous Euler's decomposition formula in \cite{Euler1776}. More than two hundred years later, these objects were generalized to the so-called \emph{generalized Euler sums} by Flajolet and Salvy \cite{FS1998}:
\begin{align}\label{Defn-Gen-Euler-Sums}
{S_{{\bf p};q}} := \sum\limits_{n = 1}^\infty  {\frac{{H_n^{\left( {{p_1}} \right)}H_n^{\left( {{p_2}} \right)} \cdots H_n^{\left( {{p_r}} \right)}}}
{{{n^q}}}},
\end{align}
where ${\bf p}=(p_1,p_2,\ldots,p_r)\in \N^r$ with $p_1\leq p_2\leq \cdots\leq p_r$ and $q\in \N\setminus \{1\}$. The quantity $w:={p _1} +  \cdots  + {p _r} + q$ is called the weight and the quantity $r$ is called the degree (order). Moreover, if $r>1$ in \eqref{Defn-Gen-Euler-Sums}, they were called \emph{nonlinear Euler sums}. As usual, repeated summands in partitions are indicated by powers, so that for instance
\[{S_{{1^3}{2^2}5;q}} = {S_{111225;q}} = \sum\limits_{n = 1}^\infty  {\frac{{H_n^3[H^{(2)} _n]^2{H^{(5)} _n}}}{{{n^q}}}}. \]

The Euler sums is in contrast to \emph{multiple zeta values} (abbr. MZVs) defined in \cite{H1992,DZ1994} as follows:
\begin{align}\label{Defn-MZV}
\zeta(\bfk)\equiv \zeta(k_1,\ldots,k_r):=\sum\limits_{n_1>\cdots>n_r>0 } \frac{1}{n_1^{k_1}\cdots n_r^{k_r}},
\end{align}
where $\bfk=(k_1,k_r,\ldots,k_r)\in \N^r$ and $k_1>1$ is ensure convergence of the series. Here $|\bfk|:=k_1+\cdots+k_r$ and $\dep(\bfk):=r$ were called the depth and the weight of MZVs, respectively. Clearly, if $r=1$ and $k_1=k$ then it becomes the Riemann zeta value $\zeta(k)\ (k\in \N\setminus \{1\})$. This theory indeed dates back to Hoffman \cite{H1992} and Zagier \cite{DZ1994} (independently at almost the same time), while recent research on this topic
has been quite active. For instance, these quantities have appeared in several areas of mathematics and physics, and have a remarkable
depth of algebraic structure in the past three decades (for example, see the book by Zhao \cite{Z2016}). Recently, several variants of classical multiple zeta values of level 2 called multiple $t$-values (abbr. MtVs), multiple $T$-values (abbr. MTVs) and multiple mixed values (abbr. MMVs) were introduced and studied in Hoffman \cite{H2019}, Kaneko-Tsumura \cite{KTA2019} and Xu-Zhao \cite{XZ2020}. It is clear that every MMV (MtV or MTV) can be written as a $\Q$-linear combination of colored MZVs of level two. The colored MZV (abbr. CMZV) of level $N$ is defined for any $(k_1,\dotsc, k_r)\in\N^r$ and $N$th roots of unity $\eta_1,\dotsc,\eta_r$ by (see Yuan-Zhao \cite{YuanZh2014a})
\begin{equation*}
Li_{k_1,\dotsc, k_r}(\eta_1,\dotsc,\eta_r):=\sum\limits_{n_1>\cdots>n_r>0}
\frac{\eta_1^{n_1}\dots \eta_r^{n_r}}{n_1^{k_1} \dots n_r^{k_r}}
\end{equation*}
which converges if $(k_1,\eta_1)\ne (1,1)$ (see \cite[Ch. 15]{Z2016}), in which case we call $(\bfk,\bfeta)$ \emph{admissible}.
The study found that there should be rich connections between Euler sums and multiple zeta values. However, we do not pursue any study of MZVs' aspects in this paper.

For an early introduction and study on the evaluations of Euler sums, the readers may consult in Bailey-Borwein-Girgensohn's \cite{BBG1994} and Flajolet-Salvy's paper \cite{FS1998}, in which they have developed experimental method and a contour integral representation approach to the evaluation of Euler sums, respectively. For
some recent progress, the readers are referred to \cite{M2014,W2017,XW2018} and references therein. Recently, some parametric Euler sums were introduced and studied, see \cite{BBD2008,Xu2017-JMAA,XC2019} and references therein. For example, in \cite[Thm. 1]{BBD2008}, Borweins and Bradley proved the results that
\begin{align*}
\sum_{n=1}^\infty \frac{1}{n(n-a)}\sum_{k=1}^{n-1} \frac{1}{k-a}=\sum_{n=1}^\infty \frac{1}{n(n-a)}\quad (a\notin \N^-).
\end{align*}
If setting $x=0$, then we obtain the well-known identity $\zeta(2,1)=\zeta(3)$. In \cite{Xu2017-JMAA}, the last author shown the formula ($a\notin \Z\setminus \{0\}$)
\begin{align*}
\frac{5}{2}\sum\limits_{n = 1}^\infty  {\frac{1}{{{{\left( {{n^2} - {a^2}} \right)}^2}}}}  - {\left( {\sum\limits_{n = 1}^\infty  {\frac{1}{{{n^2} - {a^2}}}} } \right)^2} = 2{a^2}\left\{ {\sum\limits_{n = 1}^\infty  {\frac{1}{{{n^2} - {a^2}}}} \sum\limits_{n = 1}^\infty  {\frac{1}{{{{\left( {{n^2} - {a^2}} \right)}^2}}}}  - \sum\limits_{n = 1}^\infty  {\frac{1}{{{{\left( {{n^2} - {a^2}} \right)}^3}}}} } \right\}.
\end{align*}
If setting $x=0$ gives $\zeta(2)^2=\frac{5}{2}\zeta(4)$. Therefore, it is possible to obtain many classical results of Euler sums and MZVs by studying the parametric Euler sums. In this paper, we will use the approach of contour integral representation to study the following (alternating) parametric Euler sums involving generalized harmonic numbers
\begin{align}\label{Defn-para--Gen-Euler-Sums}
{S^{\sigma}_{{\bf p};{\bf q}}}(a_1,a_2,\ldots,a_m) := \sum\limits_{n = 1}^\infty  {\frac{{H_n^{\left( {{p_1}} \right)}H_n^{\left( {{p_2}} \right)} \cdots H_n^{\left( {{p_r}} \right)}\sigma^n}}
{(n+a_1)^{q_1}(n+a_2)^{q_2}\cdots (n+a_m)^{q_m}}}\quad (a_1,\ldots,a_m\notin \N^-),
\end{align}
where $\sigma\in \{\pm 1\},{\bf p}=(p_1,\ldots,p_r)\in \N^r$ and ${\bf q}=(q_1,\ldots,q_m)\in \N^m_0$ with $q_1+\cdots+q_m\geq 2$. Obviously, if $\sigma=1,m=1$ and $a_1=0,q_1=q$ in \eqref{Defn-para--Gen-Euler-Sums}, then it becomes the classical Euler sums $S_{{\bf p};q}$ \eqref{Defn-Gen-Euler-Sums}.

\section{Main Results}

We define a complex kernel function $\xi(z)$ with two requirements: (i). $\xi(z)$ is meromorphic in the whole complex plane. (ii). $\xi(z)$ satisfies $\xi (z)=o(z)$ over an infinite collection of circles $|z|=\rho_k$ with $\rho_k\to \infty $. Applying these two conditions of kernel function $\xi(z)$, Flajolet and Salvy discovered the following residue lemma.
\begin{lem}\emph{(cf. \cite{FS1998})}\label{lem-residue}
Let $\xi(z)$ be a kernel function and let $r(z)$ be a rational function which is $O(z^{-2})$ at infinity. Then
\begin{align*}
\sum_{\alpha\in E} \Res(r(z)\xi(z),\alpha)+ \sum_{\beta\in S} \Res(r(z)\xi(z),\beta) = 0,
\end{align*}
where $S$ is the set of poles of $r(z)$ and $E$ is the set of poles of $\xi(z)$ that are not poles of $r(z)$. Here $\Res(r(z),\alpha)$ denotes the residue of $r(z)$ at $z=\alpha$.
\end{lem}

Furthermore, Flajolet and Salvy \cite[Eq. (2.4)]{FS1998} found the fact that any polynomial form in $\pi \cot\pi z,\ \frac{\pi}{\sin \pi z},\ \psi^{(j)}(\pm z)$ is itself a kernel function with poles at a subset of the integers. Here, ${\psi ^{\left( j \right)}}\left( z \right)$ stands for the polygamma function of order $j$ defined as the $(j+1)$-st derivative of the logarithm of the gamma function:
\[{\psi ^{\left( j \right)}}\left( z \right): = \frac{{{d^j}}}{{d{z^j}}}\psi \left( z \right) = \frac{{{d^{j+1}}}}{{d{z^{j+1}}}}\log \Gamma \left( z \right).\]
Thus
\[{\psi ^{\left( 0 \right)}}\left( z \right) = \psi \left( z \right) = \frac{{\Gamma '\left( z \right)}}{{\Gamma \left( z \right)}}\]
holds, where $\psi (x)$ is the digamma function and $\Gamma \left( z \right)$ is the usual gamma function. Observe that ${\psi ^{\left( j \right)}}\left( z \right)$ satisfy the following relations
\[\psi \left( z \right) =  - \gamma  + \sum\limits_{n = 0}^\infty  {\left( {\frac{1}{{n + 1}} - \frac{1}{{n + z}}} \right)} ,\;z\notin  \N^-_0, \]
\[{\psi ^{\left( j \right)}}\left( z \right) = {\left( { - 1} \right)^{j+ 1}}j!\sum\limits_{k = 0}^\infty  {1/{{\left( {z + k} \right)}^{j + 1}}},\ j\in \N,\]
\[\psi \left( {z + n} \right) = \frac{1}{z} + \frac{1}{{z + 1}} +  \cdots  + \frac{1}{{z + n - 1}} + \psi \left( z \right),\;n \in \N .\]
where, $\gamma$ denotes the Euler-Mascheroni constant, defined by
\[\gamma  := \mathop {\lim }\limits_{n \to \infty } \left( {\sum\limits_{k = 1}^n {\frac{1}{k}}  - \log n} \right) =  - \psi \left( 1 \right) \approx {\rm{ 0 }}{\rm{. 577215664901532860606512 }}....\]
Moreover, from classical expansions and the properties of $\psi$ function, they listed the following expressions of $\pi \cot\pi z$ and $\psi^{(j)}(- z)$ at an integer $n$.

\begin{lem}\emph{(cf. \cite{FS1998})}\label{lem-approach-psi} For integer $p$,
\begin{align}
&\pi \cot(\pi z)\overset{z\rightarrow n}{=}\frac{1}{z-n}-2\sum\limits_{k=1}^{\infty}\zeta(2k)(z-n)^{2k-1}\quad (n\in\Z),\\
&\frac{\pi}{\sin(\pi z)}\overset{z\rightarrow n}{=}(-1)^n \left(\frac1{z-n}+2\sum_{k=1}^\infty {\bar \zeta}(2k)(z-n)^{2k-1}\right)\quad (n\in\Z),\\
&\frac{\psi^{(p-1)}(-z)}{(p-1)!}
\overset{z\rightarrow n}{=}\frac{1}{(z-n)^{p}}\sum\limits_{k=1}^{\infty}\binom{k+p-2}{p-1}
\left[(-1)^{p}\zeta(k+p-1)-(-1)^{k}H^{(k+p-1)}_{n}\right]\left(z-n\right)^{k+p-1}\nonumber\\
&\quad\quad\quad\quad\quad\quad+\frac{1}{(z-n)^{p}}\quad (n\in\N_0),\\
&\frac{\psi^{(p-1)}(-z)}{(p-1)!}\overset{z\rightarrow -n}{=}(-1)^{p}\sum\limits_{k=1}^{\infty}\binom{p+k-2}{p-1}
 \left[\zeta(p+k-1)-H^{(p+k-1)}_{n-1}\right](z+n)^{k-1}\quad (n\in \N),
\end{align}
where $\ze(1)$ should be interpreted as $0$ and if $p=1$, replace $\psi(-z)$ by $\psi(-z)+\gamma$. ${\bar \zeta}(s)$ denotes the \emph{alternating Riemann zeta function} which is defined by
\begin{align}
{\bar \zeta}(s):=\sum_{n=1}^\infty \frac{(-1)^{n-1}}{n^s}\quad (\Re(s)\geq 1).
\end{align}
\end{lem}

In \cite{FS1998}, Flajolet and Salvy used residue computations on large circular contour and specific functions to obtain more independent relations for classical Euler sums. These functions are of the form $\xi(z)r(z)$, where $r(z):=1/{z^q}$ and $\xi(z)$ is a product of cotangent (or cosecant) and polygamma function. In \cite{Xu2017-JMAA}, the last author used the method of Flajolet and Salvy to obtain some explicit evaluations of parametric Euler sums. Next, we will use a similar method with the help of Lemmas \ref{lem-residue} and \ref{lem-approach-psi} to give some explicit evaluations of the parametric Euler sums \eqref{Defn-para--Gen-Euler-Sums}.

\subsection{Linear Parametric Euler Sums}

In this subsection, we apply the contour integral representation approach to consider the linear parametric Euler sums
\[\sum_{n=1}^\infty \frac{H_n^{(p)}}{(n+a)(n+b)}\ \text{and}\ \sum_{n=1}^\infty \frac{H_n^{(p)}}{(n+a)(n+b)}(-1)^n.\]

\begin{thm}\label{thm-para2-harmonic-double-sum} For positive integer $p$ and complexes $a,b$ with $a\neq b$ and $a,b\notin \Z$, we have
\begin{align}\label{para2-harmonic-double-sum}
&\sum_{n=1}^\infty \frac{H_n^{(p)}}{(n+a)(n+b)}-(-1)^p \sum_{n=1}^\infty \frac{H_{n-1}^{(p)}}{(n-a)(n-b)}\nonumber\\
&=2\frac{(-1)^p}{b-a} \sum_{k=0}^{[p/2]} \ze(2k)\Big(\ze(p-2k+1;a)-\ze(p-2k+1;b)\Big)\nonumber\\
&\quad+\frac{(-1)^p}{b-a} \Big(\pi \cot(\pi a)\big(\ze(p;a)-\ze(p)\big)-\pi \cot(\pi b) \big(\ze(p;b)-\ze(p)\big)\Big),
\end{align}
where $\ze(0):=-1/2$ and $\ze(1;a):=-(\psi(a)+\gamma)$, and $\ze(s;a)$ is \emph{Hurwitz zeta function} defined by ($a\neq 0,-1,-2,-3,\ldots$)
\[\ze(s;a):=\sum_{n=0}^\infty \frac{1}{(n+a)^s}\quad (\Re(s)>1).\]
\end{thm}
\begin{proof}
Applying the method of contour integration in \cite{Xu2017-JMAA} (similar to the proof in \cite[Thm. 2.4]{Xu2017-JMAA}), we need to consider the contour integral
\begin{align*}
\oint\limits_{\left( \infty  \right)}F(z)dz:=\oint\limits_{\left( \infty  \right)}\frac{\pi \cot(\pi z) \psi^{(p-1)}(-z)}{(p-1)!(z+a)(z+b)}dz,
\end{align*}
where $\oint_{\left( \infty  \right)}$ denotes integratl along large circles, that is, the limit of integrals $\oint_{\left| z \right| = R}$ as $R\to \infty$.
Clearly, ${\pi \cot(\pi z)\psi^{(p-1)}(-z)}/(p-1)!$ is a kernel function. Hence, $\oint_{\left( \infty  \right)} F(z)dz=0$.
Note that the function in the contour integration only have poles at $z=n$ and $-a,-b$. Applying Lemma \ref{lem-approach-psi}, we can find that, at a nonnegative integer $n$, the pole has order $p+1$. Moreover, by a direct calculation, for $n\in \N_0$, we have
\begin{align*}
F(z)\overset{z\rightarrow n}{=}\frac{1}{(z-n)^{p+1}}\frac{1-2\sum\limits_{1\leq k\leq [p/2]}
\zeta(2k)(z-n)^{2k}+\left((-1)^{p}\zeta(p)+H_{n}^{(p)}\right)(z-n)^{p}+o((z-n)^{p})}{(z+a)(z+b)}
\end{align*}
and the residue is
\begin{align*}
\text{Res}(F(z),n)&=\lim_{z\to n}\frac{1}{p!}\frac{d^{p}}{dz^{p}}\left\{(z-n)^{p+1} F(z)\right\}\\
&=2\frac{(-1)^p}{a-b}\sum_{k=0}^{[p/2]}\ze(2k)\left(\frac{1}{(n+a)^{p-2k+1}}-\frac{1}{(n+b)^{p-2k+1}}\right)+\frac{(-1)^p\ze(p)+H_n^{(p)}}{(n+a)(n+b)}.
\end{align*}
At a negative integer $-n$ and complexes $-a,-b$, the poles are simple and residues are
\begin{align*}
&\text{Res}(F(z),-n)=(-1)^p \frac{\ze(p)-H_{n-1}^{(p)}}{(n-a)(n-b)},\\
&\text{Res}(F(z),-a)=(-1)^p \frac{\pi \cot(\pi a)}{a-b} \ze(p;a),\\
&\text{Res}(F(z),-b)=(-1)^p \frac{\pi \cot(\pi b)}{b-a} \ze(p;b).
\end{align*}
Hence, using Lemma \ref{lem-residue}, we have
\begin{align*}
\sum_{n=0}^\infty \text{Res}(F(z),n)+\sum_{n=1}^\infty \text{Res}(F(z),-n)+\text{Res}(F(z),-a)+\text{Res}(F(z),-b)=0.
\end{align*}
Summing these four contributions yields the statement of the theorem.
\end{proof}

\begin{cor}\emph{(cf. \cite{BBD2008,Xu2017-JMAA})} For $a\notin \Z$ and $m\in \N_0$,
\begin{align}\label{for-para-harm-num-hur-ze}
\sum\limits_{n=1}^{\infty}\frac{H_{n}^{(2m+1)}}{n^{2}-a^{2}}
&=\frac{1}{2}\sum\limits_{n=1}^{\infty}\frac{1}{n^{2m+1}(n^{2}-a^{2})}
+\frac{1}{2a}\sum\limits_{k=0}^{m}\zeta(2k)\left(\zeta(2m-2k+2;a)-\zeta(2m-2k+2;-a)\right)\nonumber\\
&\quad+\frac{1}{4a}\pi \cot(\pi a)\Big(\zeta(2m+1;a)+\zeta(2m+1;-a)-2\zeta(2m+1)\Big).
\end{align}
\end{cor}
\begin{proof}
Setting $b=-a$ and $p=2m+1$ in \eqref{para2-harmonic-double-sum} yields the desired result.
\end{proof}

\begin{cor} For integer $m\geq 0$ and $a\notin\Z$ with $a\neq 1/2$,
\begin{align}\label{para1-harmonic-double-sum}
\sum_{n=1}^\infty \frac{H_n^{(2m+1)}}{(n+a)(n+1-a)}&=\frac{1}{2a-1} \sum_{k=0}^m \ze(2k)\Big(\ze(2m-2k+2;a)-\ze(2m-2k+2;1-a)\Big)\nonumber\\
&\quad+\frac{\pi \cot(\pi a)}{2(2a-1)} \Big(\ze(2m+1;a)+\ze(2m+1;1-a)-2\ze(2m+1)\Big).
\end{align}
\end{cor}
\begin{proof}
Setting $b=1-a$ and $p=2m+1$ in \eqref{para2-harmonic-double-sum} yields the desired result.
\end{proof}

It is obvious that upon differentiating both members of \eqref{para2-harmonic-double-sum} $k-1$ times with
respect to $a$ and $l-1$ times with respect to $b$, we obtain an explicit evaluation of the combined series
\[\sum_{n=1}^\infty \frac{H_n^{(p)}}{(n+a)^k(n+b)^{l}}-(-1)^{p+k+l}\sum_{n=1}^\infty \frac{H_{n-1}^{(p)}}{(n-a)^k(n-b)^l}\quad (k,l\geq 1).\]
For example, upon differentiating both members of \eqref{para2-harmonic-double-sum} $1$ times with
respect to $a$ and $1$ times with respect to $b$, and noting the facts that $\psi^{(j)}(a)=(-1)^{j+1}j!\zeta(j+1;a)$ and $\ze(1;a):=-(\psi(a)+\gamma)$, by a direct calculation,
we deduce
\begin{align}
&\sum_{n=1}^\infty \frac{H_n}{(n+a)^2(n+b)^2}+\sum_{n=1}^\infty \frac{H_{n-1}}{(n-a)^2(n-b)^2}\nonumber\\
&=-\frac{2 (-\psi ^{(1)}(a)-\pi  \cot (\pi  a) (\psi(a)+\gamma )+\psi ^{(1)}(b)+\pi  \cot (\pi  b) (\psi(b)+\gamma ))}{(a-b)^3}\nonumber\\&\quad-\frac{\psi ^{(2)}(a)+\pi  \cot (\pi  a) \psi ^{(1)}(a)-\pi ^2 \csc ^2(\pi  a) (\psi (a)+\gamma )}{(a-b)^2}\nonumber\\&\quad-\frac{\psi ^{(2)}(b)+\pi  \cot (\pi  b) \psi ^{(1)}(b)-\pi ^2 \csc ^2(\pi  b) (\psi(b)+\gamma )}{(a-b)^2}.
\end{align}
Then setting $b=1-a$ yields the following evaluation
\begin{align}
&\sum_{n=1}^\infty \frac{H_n}{(n+a)^2(n+1-a)^2}\nonumber\\
&=\frac{-\psi ^{(1)}(1-a)+\psi ^{(1)}(a)+\pi  \cot (\pi  a) (\psi(1-a)+\gamma)+\pi  \cot (\pi  a) (\psi (a)+\gamma)}{(2 a-1)^3}\nonumber\\
&\quad-\frac{\psi ^{(2)}(1-a)-\pi  \cot (\pi  a) \psi ^{(1)}(1-a)-\pi ^2 \csc ^2(\pi  a) (\psi (1-a)+\gamma )}{2(2 a-1)^2}\nonumber\\
&\quad-\frac{\psi ^{(2)}(a)+\pi  \cot (\pi  a) \psi ^{(1)}(a)-\pi ^2 \csc ^2(\pi  a) (\psi (a)+\gamma )}{2(2 a-1)^2}.
\end{align}

More general, we can get the following general theorem.
\begin{thm}\label{thm-para1-ration-funct-residue} For $p\in \N$,
\begin{align}\label{para1-ration-funct-residue}
&\sum_{n=0}^\infty \Big((-1)^p\zeta(p)+H_n^{(p)}\Big)r(n)+(-1)^p\sum_{n=1}^\infty \Big(\zeta(p)-H_{n-1}^{(p)}\Big)r(-n)\nonumber\\
&-2\sum_{k=0}^{[p/2]} \frac1{(p-2k)!} \ze(2k) \sum_{n=0}^\infty r^{(p-2k)}(n)+\sum_{\beta\in S}{\rm Res}(f(z),\beta)=0,
\end{align}
where $\ze(0)$ and $\ze(1)$ should be interpreted as $-1/2$ and $0$ wherever they
occur. $r^{(p)}(n)$ is defined as the $p$-st derivative of $r(z)$ with $z=n$, $r(z)$ is a rational function which is $O(z^{-2})$ at infinity, $S$ is the set of poles of $r(z)$ and any integers $n$ are not poles of $r(z)$. The function $f(z)$ is defined by
\[f(z):=\frac{\pi \cot(\pi z) \psi^{(p-1)}(-z)}{(p-1)!}r(z).\]
\end{thm}
\begin{proof} Considering the contour integral
\begin{align*}
\oint\limits_{\left( \infty  \right)}f(z)dz=0.
\end{align*}
By a similar argument as in the proof of Theorem \ref{thm-para2-harmonic-double-sum}, we obtain
\begin{align*}
&\text{Res}(f(z),n)=-2\sum_{k=0}^{[p/2]}\ze(2k)\frac{r^{(p-2k)}(n)}{(p-2k)!}+\Big((-1)^p\ze(p)+H_n^{(p)}\Big)r(n)\quad (n\in \N_0),\\
&\text{Res}(f(z),-n)=(-1)^p\Big(\zeta(p)-H_{n-1}^{(p)}\Big)r(-n)\quad (n\in \N).
\end{align*}
Hence, applying Lemma \ref{lem-residue}, we have
\begin{align*}
\sum_{n=0}^\infty \text{Res}(f(z),n)+\sum_{n=1}^\infty \text{Res}(f(z),-n)+\sum_{\beta\in S}{\rm Res}(f(z),\beta)=0.
\end{align*}
Thus, combining related identities yields the desired result.
\end{proof}

It is clear that the Theorem \ref{thm-para2-harmonic-double-sum} follows immediately from Theorem \ref{thm-para1-ration-funct-residue} if we set $r(z)=1/((z+a)(z+b))$.

\begin{thm}\label{thm-para2-alter-harmonic-double-sum} For positive integer $p$ and complexes $a,b$ with $a\neq b$ and $a,b\notin \Z$, we have
\begin{align}\label{para2-alter-harmonic-double-sum}
&\sum_{n=1}^\infty \frac{H_n^{(p)}}{(n+a)(n+b)}(-1)^n-(-1)^p \sum_{n=1}^\infty \frac{H_{n-1}^{(p)}}{(n-a)(n-b)}(-1)^n\nonumber\\
&=2\frac{(-1)^p}{b-a} \sum_{k=0}^{[p/2]} {\bar \ze}(2k)\Big({\bar \ze}(p-2k+1;b)-{\bar \ze}(p-2k+1;a)\Big)\nonumber\\
&\quad+\frac{(-1)^p}{b-a}\pi \left( \frac{\ze(p;a)-\ze(p)}{\sin(\pi a)}-\frac{\ze(p;b)-\ze(p)}{\sin(\pi b)}\right),
\end{align}
where ${\bar \ze}(0):=1/2$ and ${\bar \ze}(s;a)$ is \emph{alternating Hurwitz zeta function} defined by ($a\notin \N_0^-$)
\[{\bar \ze}(s;a):=\sum_{n=0}^\infty \frac{(-1)^n}{(n+a)^s}\quad (\Re(s)\geq 1).\]
\end{thm}
\begin{proof}
The proof of Theorem \ref{thm-para2-alter-harmonic-double-sum} is similar to the proof of Theorem \ref{thm-para2-harmonic-double-sum}, we need to consider the contour integral
\begin{align*}
\oint\limits_{\left( \infty  \right)}G(z)dz:=\oint\limits_{\left( \infty  \right)}\frac{\pi  \psi^{(p-1)}(-z)}{(p-1)!\sin(\pi z)(z+a)(z+b)}dz.
\end{align*}
Clearly, ${\pi \psi^{(p-1)}(-z)}/((p-1)!\sin(\pi z))$ is a kernel function. Hence, $\oint_{\left( \infty  \right)} G(z)dz=0$.
Note that the function in the contour integral only have poles at $z=n$ and $-a,-b$. Applying Lemma \ref{lem-approach-psi}, we can find that, at a nonnegative integer $n$, the pole has order $p+1$. Moreover, by a direct calculation, for $n\in \N_0$, we have
\begin{align*}
G(z)\overset{z\rightarrow n}{=}\frac{(-1)^n}{(z-n)^{p+1}}\frac{1+2\sum\limits_{1\leq k\leq [p/2]}
{\bar \zeta}(2k)(z-n)^{2k}+\left((-1)^{p}\zeta(p)+H_{n}^{(p)}\right)(z-n)^{p}+o((z-n)^{p})}{(z+a)(z+b)}
\end{align*}
and the residue is
\begin{align*}
\text{Res}(G(z),n)&=\lim_{z\to n}\frac{1}{p!}\frac{d^{p}}{dz^{p}}\left\{(z-n)^{p+1} G(z)\right\}\\
&=2\frac{(-1)^p}{b-a}\sum_{k=0}^{[p/2]}{\bar \ze}(2k)\left(\frac{(-1)^n}{(n+a)^{p-2k+1}}-\frac{(-1)^n}{(n+b)^{p-2k+1}}\right)+\frac{(-1)^p\ze(p)+H_n^{(p)}}{(n+a)(n+b)}(-1)^n.
\end{align*}
At a negative integer $-n$ and complexes $-a,-b$, the poles are simple and residues are
\begin{align*}
&\text{Res}(G(z),-n)=(-1)^p \frac{\ze(p)-H_{n-1}^{(p)}}{(n-a)(n-b)}(-1)^n,\\
&\text{Res}(G(z),-a)=(-1)^p \frac{\pi \ze(p;a)}{\sin(\pi a)(a-b)} ,\\
&\text{Res}(G(z),-b)=(-1)^p \frac{\pi \ze(p;b)}{\sin(\pi b)(b-a)}.
\end{align*}
Hence, using Lemma \ref{lem-residue}, we have
\begin{align*}
\sum_{n=0}^\infty \text{Res}(G(z),n)+\sum_{n=1}^\infty \text{Res}(G(z),-n)+\text{Res}(G(z),-a)+\text{Res}(G(z),-b)=0.
\end{align*}
Noting the fact that
\begin{align*}
\sum_{n=0}^\infty \frac{(-1)^n}{(n+a)(n+b)}+\sum_{n=1}^\infty \frac{(-1)^n}{(n-a)(n-b)}=\frac{\pi}{b-a} \left(\frac1{\sin (\pi a)}-\frac1{\sin (\pi b)}\right).
\end{align*}
Summing these four contributions yields the statement of the theorem.
\end{proof}

\begin{cor} For $a\notin \Z$ and $m\in \N_0$,
\begin{align}\label{for-para-alter-harm-num-hur-ze}
\sum\limits_{n=1}^{\infty}\frac{H_{n}^{(2m+1)}}{n^{2}-a^{2}}(-1)^n
&=\frac{1}{2}\sum\limits_{n=1}^{\infty}\frac{(-1)^n}{n^{2m+1}(n^{2}-a^{2})}\nonumber\\
&\quad+\frac{1}{2a}\sum\limits_{k=0}^{m}{\bar \zeta}(2k)\left({\bar \zeta}(2m-2k+2;-a)-{\bar \zeta}(2m-2k+2;a)\right)\nonumber\\
&\quad+\frac{\pi}{4a \sin(\pi a)}\Big(\zeta(2m+1;a)+\zeta(2m+1;-a)-2\zeta(2m+1)\Big).
\end{align}
\end{cor}
\begin{proof}
Setting $b=-a$ and $p=2m+1$ in \eqref{para2-alter-harmonic-double-sum} yields the desired result.
\end{proof}

\begin{cor} For integer $m\geq 1$ and $a\notin\Z$ with $a\neq 1/2$,
\begin{align}\label{para1-alter-harmonic-double-sum}
\sum_{n=1}^\infty \frac{H_n^{(2m)}}{(n+a)(n+1-a)}(-1)^n&=\frac{1}{1-2a} \sum_{k=0}^m {\bar \ze}(2k)\Big({\bar \ze}(2m-2k+2;1-a)-{\bar \ze}(2m-2k+2;a)\Big)\nonumber\\
&\quad+\frac{\pi}{2(1-2a)} \frac{\ze(2m;a)-\ze(2m;1-a)}{\sin(\pi a)}.
\end{align}
\end{cor}
\begin{proof}
Setting $b=1-a$ and $p=2m$ in \eqref{para2-alter-harmonic-double-sum} yields the desired result.
\end{proof}

Similar to Theorem \ref{thm-para1-ration-funct-residue} , we can get the following general theorem.
\begin{thm}\label{thm-alter-para1-ration-funct-residue} For $p\in \N$,
\begin{align}\label{a;lter-para1-ration-funct-residue}
&\sum_{n=0}^\infty \Big((-1)^p\zeta(p)+H_n^{(p)}\Big)(-1)^nr(n)+(-1)^p\sum_{n=1}^\infty \Big(\zeta(p)-H_{n-1}^{(p)}\Big)(-1)^nr(-n)\nonumber\\
&+2\sum_{k=0}^{[p/2]} \frac1{(p-2k)!} {\bar \ze}(2k) \sum_{n=0}^\infty (-1)^nr^{(p-2k)}(n)+\sum_{\beta\in T}{\rm Res}(g(z),\beta)=0,
\end{align}
where $\ze(0)$ and $\ze(1)$ should be interpreted as $-1/2$ and $0$ wherever they
occur. $r^{(p)}(n)$ is defined as the $p$-st derivative of $r(z)$ with $z=n$, $r(z)$ is a rational function which is $O(z^{-2})$ at infinity, $T$ is the set of poles of $r(z)$ and any integers $n$ are not poles of $r(z)$. The function $g(z)$ is defined by
\[g(z):=\frac{\pi  \psi^{(p-1)}(-z)}{(p-1)!\sin(\pi z)}r(z).\]
\end{thm}
\begin{proof} Considering the contour integral
\begin{align*}
\oint\limits_{\left( \infty  \right)}g(z)dz=0.
\end{align*}
By a similar argument as in the proof of Theorem \ref{thm-para2-alter-harmonic-double-sum}, we obtain
\begin{align*}
&\text{Res}(g(z),n)=2\sum_{k=0}^{[p/2]}{\bar \ze}(2k)\frac{r^{(p-2k)}(n)}{(p-2k)!}(-1)^n+\Big((-1)^p\ze(p)+H_n^{(p)}\Big)(-1)^nr(n)\quad (n\in \N_0),\\
&\text{Res}(g(z),-n)=(-1)^p\Big(\zeta(p)-H_{n-1}^{(p)}\Big)(-1)^nr(-n)\quad (n\in \N).
\end{align*}
Hence, applying Lemma \ref{lem-residue}, we have
\begin{align*}
\sum_{n=0}^\infty \text{Res}(g(z),n)+\sum_{n=1}^\infty \text{Res}(g(z),-n)+\sum_{\beta\in T}{\rm Res}(g(z),\beta)=0.
\end{align*}
Thus, combining related identities yields the desired result.
\end{proof}

Hence, Theorem \ref{thm-para2-alter-harmonic-double-sum} follows immediately from Theorem \ref{thm-alter-para1-ration-funct-residue} if we set $r(z)=1/((z+a)(z+b))$.

In \cite{BBD2008}, Borweins and Bradley used the evaluation \eqref{for-para-harm-num-hur-ze} to obtain an explicit formula of double zeta values $\zeta(2j,2m+1)\ (j\in \N,m\in \N_0)$ by using power series expansions and comparing the coefficients on both sides. Similarly, applying \eqref{for-para-alter-harm-num-hur-ze}, we can also get the following corollary.
\begin{cor} For $j\in \N$ and $m\in \N_0$,
\begin{align}\label{double-alter-zeta-values}
\zeta({\overline{2j}},2m+1)&=\frac1{2}{\bar \ze}(2m+2j+1)-\sum_{k=0}^m \binom{2j+2m-2k}{2j-1}{\bar \ze}(2k){\bar \ze}(2m+2j+1-2k)\nonumber\\
&\quad+\sum_{l=0}^{j-1} \binom{2j+2m-2l}{2m}{\bar \ze}(2l)\ze(2m+2j+1-2l),
\end{align}
where $\zeta({\overline{2j}},2m+1)$ is a alternating double zeta values defined by
\begin{align}
\zeta({\overline{2j}},2m+1):=\sum_{n>k\geq 1} \frac{(-1)^n}{n^{2j}k^{2m+1}}=\sum_{n=1}^\infty \frac{H_{n-1}^{(2m+1)}}{n^{2j}}(-1)^n.
\end{align}
\end{cor}
\begin{proof}
By direct calculations, for $|a|<1$, we have
\begin{align*}
&\sum_{n=1}^\infty \frac{H_{n-1}^{(2m+1)}}{n^2-a^2}(-1)^n=\sum_{j=1}^\infty \zeta(\overline{2j},2m+1)a^{2j-2},\\
&\sum_{n=1}^\infty \frac{(-1)^n}{n^{2m+1}(n^2-a^2)}=-\sum_{j=1}^\infty {\bar \ze}(2m+2j+1)a^{2j-2},\\
&\frac{\pi}{\sin(\pi a)}=\frac1{a}+2a\sum_{j=1}^\infty {\bar \ze}(2j)a^{2j-2},\\
&\ze(2m+1;a)+\ze(2m+1;1-a)-2\ze(2m+1)\\&\quad=2\sum_{j=1}^\infty \binom{2j+2m}{2j}\ze(2m+2j+1)a^{2j},\\
&{\bar \ze}(2m-2k+2;-a)-{\bar \ze}(2m-2k+2;a)\\&\quad=-2\sum_{j=1}^\infty \binom{2j+2m-2k}{2j-1}{\bar \ze}(2m+2j+1-2k)a^{2j-1}.
\end{align*}
Then, using \eqref{for-para-alter-harm-num-hur-ze} gives
\begin{align*}
&\sum_{j=1}^\infty \ze(\overline{2j},2m+1)a^{2j-2}\\&=\frac1{2} \sum_{j=1}^\infty {\bar \ze}(2m+2j+1)a^{2j-2}-\sum_{j=1}^\infty \left\{\sum_{k=0}^m \binom{2j+2m-2k}{2j-1}{\bar \ze}(2k){\bar \ze}(2m+2j+1-2k)\right\}a^{2j-2}\\
&\quad+\frac1{2} \sum_{j=1}^\infty \binom{2j+2m}{2j} \ze(2m+2j+1)a^{2j-2}+\sum_{j=1}^\infty \sum_{j_1+j_2=j,\atop j_1,j_2\geq 1} \binom{2j_2+2m}{2j_2}{\bar \ze}(2j_1)\ze(2m+2j_2+1)a^{2j-2}.
\end{align*}
Thus, comparing the coefficients of $a^{2j-2}$ in above identity yields the desired evaluation.
\end{proof}

\subsection{Quadratic Parametric Euler Sums}
Now, we give some evaluations of quadratic parametric Euler sums by using the method of contour integral.

\begin{thm}\label{thm-quadratic-para-ES} For $p,m\in \N$ and $a,b\notin \Z$ with $a\neq b$,
\begin{align}\label{Formu-quadratic-para-ES}
&\sum\limits_{n=1}^{\infty}\frac{H_{n}^{(p)}H_{n}^{(m)}}{(n+a)(n+b)}
+(-1)^{p+m}\sum\limits_{n=1}^{\infty}\frac{H_{n-1}^{(p)}H_{n-1}^{(m)}}{(n-a)(n-b)}\nonumber\\
&+(-1)^{m}\zeta(m)\sum\limits_{n=1}^{\infty}\frac{H_{n}^{(p)}}{(n+a)(n+b)}
+(-1)^{p}\zeta(p)\sum\limits_{n=1}^{\infty}\frac{H_{n}^{(m)}}{(n+a)(n+b)}\nonumber\\
&-(-1)^{p+m}\zeta(m)\sum\limits_{n=1}^{\infty}\frac{H_{n-1}^{(p)}}{(n-a)(n-b)}
-(-1)^{p+m}\zeta(p)\sum\limits_{n=1}^{\infty}\frac{H_{n-1}^{(m)}}{(n-a)(n-b)}\nonumber\\
&+(-1)^{p+m}\frac{\pi \cot(\pi b)}{b-a}\left\{\zeta(p;b)\zeta(m;b)-\zeta(p)\zeta(m)\right\}\nonumber\\
&-(-1)^{p+m}\frac{\pi \cot(\pi a)}{b-a}\left\{\zeta(p;a)\zeta(m;a)-\zeta(p)\zeta(b)\right\}\nonumber\\
&-2\frac{(-1)^{p+m}}{b-a}\sum\limits_{k=0}^{[(p+m)/{2}]}\zeta(2k)
\left\{\zeta(p+m-2k+1;a)-\zeta(p+m-2k+1;b)\right\}\nonumber\\
&+2\frac{(-1)^{p+m}}{b-a}
\sum\limits_{2k_{1}+k_{2}\leq m+1,\atop k_{1}\geq 0,k_{2}\geq 1}
(-1)^{k_{2}}\binom{k_{2}+p-2}{p-1}\zeta(2k_{1})\nonumber\\&\quad\quad\quad\quad\quad\quad\quad\times
\left\{ \begin{array}{l}
\zeta(k_{2}+p-1)\left[\zeta(m-2k_{1}-k_{2}+2;a)-\zeta(m-2k_{1}-k_{2}+2;b)\right]\nonumber\\
-(-1)^{p+k_{2}}\sum\limits_{n=0}^{\infty}
\left[\frac{H_{n}^{(k_{2}+p-1)}}{(n+a)^{m-2k_{1}-k_{2}+2}}-\frac{H_{n}^{(k_{2}+p-1)}}{(n+b)^{m-2k_{1}-k_{2}+2}}\right]\nonumber\\
 \end{array} \right\}\\
&+2\frac{(-1)^{p+m}}{b-a}\sum\limits_{2k_{1}+k_{2}\leq p+1,\atop k_{1}\geq 0,k_{2}\geq 1}
(-1)^{k_{2}}\binom{k_{2}+m-2}{m-1}\zeta(2k_{1})\nonumber\\&\quad\quad\quad\quad\quad\quad\quad\times
\left\{ \begin{array}{l}
\zeta(k_{2}+m-1)\left[\zeta(p-2k_{1}-k_{2}+2;a)-\zeta(p-2k_{1}-k_{2}+2;b)\right]\nonumber\\
-(-1)^{m+k_{2}}\sum\limits_{n=0}^{\infty}
\left[\frac{H_{n}^{(k_{2}+m-1)}}{(n+a)^{p-2k_{1}-k_{2}+2}}-\frac{H_{n}^{(k_{2}+m-1)}}{(n+b)^{p-2k_{1}-k_{2}+2}}\right]\end{array} \right\}\nonumber\\
&=0,
\end{align}
where $\zeta \left(1\right)$ and $\zeta(0)$ should be interpreted as $0$ and $-1/2$, respectively, wherever they occur. $\ze(1;a):=-(\psi(a)+\gamma)$.
\end{thm}
\begin{proof}
 Similarly to the proof of Theorem \ref{thm-para2-harmonic-double-sum}, we consider the contour integral
\begin{align*}
\oint\limits_{\left( \infty  \right)}H(z)dz:=\oint\limits_{\left( \infty  \right)}\frac{\pi \cot(\pi z)\psi^{(p-1)}(-z)\psi^{(m-1)}(-z)}{(z+a)(z+b)(p-1)!(m-1)!}dz=0.
\end{align*}
Observe that $H(z)$ has poles only at $-a,-b$ and the integers. Applying Lemma \ref{lem-approach-psi}, we can find that, at a nonnegative integer $n$, the pole has order $p+m+1$. Moreover, by a straightforward calculation, for $n\in \N_0$, we have
\begin{align*}
&H(z)\overset{z\rightarrow n}{=}\frac{1}{(z-n)^{p+m+1}}\frac{1}{(z+a)(z+b)}\\
&\quad\quad\times\left\{\begin{array}{l}
1-2\sum\limits_{k=1}^{[(p+m)/{2}]}\zeta(2k)(z-n)^{2k}\\+\sum\limits_{k=1}^{m+1}\binom{k+p-2}{p-1}\left[(-1)^{p}\zeta(k+p-1)-(-1)^{k}H_{n}^{(k+p-1)}\right](z-n)^{k+p-1}\\
-2\sum\limits_{2k_{1}+k_{2}\leq m+1,\atop k_{1},k_{2}\geq 1}
\binom{k_{2}+p-2}{p-1}\zeta(2k_{1})\left[(-1)^{p}\zeta(k_{2}+p-1)-(-1)^{k_{2}}H_{n}^{(k_{2}+p-1)}\right]\\ \quad\quad\quad\quad\quad\quad\quad\quad\quad\times(z-n)^{2k_{1}+k_{2}+p-1}\\
+\sum\limits_{k=1}^{p+1}\binom{k+m-2}{m-1}
\left[(-1)^{m}\zeta(k+m-1)-(-1)^{k}H_{n}^{(k+m-1)}\right](z-n)^{k+m-1}\\
-2\sum\limits_{2k_{1}+k_{2}\leq p+1,\atop k_{1},k_{2}\geq 1}
\binom{k_{2}+m-2}{m-1}\zeta(2k_{1})\left[(-1)^{m}\zeta(k_{2}+m-1)-(-1)^{k_{2}}H_{n}^{(k_{2}+m-1)}\right]\\ \quad\quad\quad\quad\quad\quad\quad\quad\quad\times(z-n)^{2k_{1}+k_{2}+m-1}\\
+\left[(-1)^{p}\zeta(p)+H_{n}^{(p)}\right]\left[(-1)^{m}\zeta(m)+H_{n}^{(m)}\right]
(z-n)^{p+m}+o((z-n)^{p+m})\end{array}
\right\}
\end{align*}
and the residue is
\begin{align*}
&\text{Res}[H(z),n]=\frac{1}{(p+m)!}\lim_{z\to n}\frac{d^{p+m}}{dz^{p+m}}\left\{(z-n)^{p+m+1}H(z)\right\}\\
&=\frac{(-1)^{p+m}}{b-a}\left\{\frac{1}{(n+a)^{p+m+1}}-\frac{1}{(n+b)^{p+m+1}}\right\}\\
&\quad-2\frac{(-1)^{p+m}}{b-a}\sum\limits_{k=1}^{[(p+m)/{2}]}\zeta(2k)\left\{\frac{1}{(n+a)^{p+m-2k+1}}-\frac{1}{(n+b)^{p+m-2k+1}}\right\}\\
&\quad-\frac{(-1)^{p+m}}{b-a}\sum\limits_{k=1}^{m+1}
\binom{k+p-2}{p-1}\left[(-1)^{k}\zeta(k+p-1)-(-1)^{p}H_{n}^{(k+p-1)}\right]\\
&\quad\quad\quad\quad\quad\quad\quad\times\left[\frac{1}{(n+a)^{m-k+2}}-\frac{1}{(n+b)^{m-k+2}}\right]\\
&\quad+2\frac{(-1)^{p+m}}{b-a}
\sum\limits_{2k_{1}+k_{2}\leq m+1,\atop k_{1},k_{2}\geq 1}
\binom{k_{2}+p-2}{p-1}(-1)^{2k_{1}+k_{2}}\zeta(2k_{1})\\
&\quad\quad\quad\times
\left[\zeta(k_{2}+p-1)-(-1)^{p+k_{2}}H_{n}^{(k_{2}+p-1)}\right]\left[\frac{1}{(n+a)^{m-2k_{1}-k_{2}+2}}-\frac{1}{(n+b)^{m-2k_{1}-k_{2}+2}}\right]\\
&\quad-\frac{(-1)^{p+m}}{b-a}\sum\limits_{k=1}^{p+1}\binom{k+m-2}{m-1}\left[(-1)^{k}\zeta(k+m-1)-(-1)^{m}H_n^{(k+m-1)}\right]
\\&\quad\quad\quad\quad\quad\quad\quad\times\left[\frac{1}{(n+a)^{p-k+2}}-\frac{1}{(n+b)^{p-k+2}}\right]\\
&\quad+2\frac{(-1)^{p+m}}{b-a}
\sum\limits_{2k_{1}+k_{2}\leq m+1,\atop k_{1},k_{2}\geq 1}
\binom{k_{2}+m-2}{m-1}(-1)^{2k_{1}+k_{2}}\zeta(2k_{1})\\
&\quad\quad\quad\times
\left[\zeta(k_{2}+m-1)-(-1)^{m+k_{2}}H_{n}^{(k_{2}+m-1)}\right]\left[\frac{1}{(n+a)^{p-2k_{1}-k_{2}+2}}-\frac{1}{(n+b)^{p-2k_{1}-k_{2}+2}}\right]\\
&\quad+\frac{(-1)^{p+m}\zeta(p)\zeta(m)+(-1)^{p}\zeta(p)H_{n}^{(m)}+(-1)^{m}\zeta(m)H_{n}^{(p)}+H_{n}^{(p)}H_{n}^{(m)}}{(n+a)(n+b)}.
\end{align*}
At a negative integer $-n$ and reals $-a,-b$, the poles are simple and residues are
\begin{align*}
&\text{Res}[H(z),-n]=(-1)^{p+m}\frac{\zeta(p)\zeta(m)-\zeta(p)H_{n-1}^{(m)}-\zeta(m)H_{n-1}^{(p)}+H_{n-1}^{(p)}H_{n-1}^{(m)}}{(n-a)(n-b)},\\
&\text{Res}[H(z),-a]=-(-1)^{p+m}\frac{\pi \cot(\pi a)}{b-a}\zeta(p;a)\zeta(m;a),\\
&\text{Res}[H(z),-b]=(-1)^{p+m}\frac{\pi \cot(\pi b)}{b-a}\zeta(p;b)\zeta(m;b).
\end{align*}
Hence, using Lemma \ref{lem-residue}, we have
\begin{align*}
\sum_{n=0}^\infty \text{Res}(H(z),n)+\sum_{n=1}^\infty \text{Res}(H(z),-n)+\text{Res}(H(z),-a)+\text{Res}(H(z),-b)=0.
\end{align*}
Summing these four contributions yields the statement of the theorem.
\end{proof}

Letting $(p,m)=(1,1)$ and $(2,2)$ in Theorem \ref{thm-quadratic-para-ES}, we can get the following corollaries.
\begin{cor}  For $a,b\notin \Z$ with $a\neq b$,
\begin{align}
&\sum\limits_{n=1}^{\infty}\frac{H_{n}^{2}}{(n+a)(n+b)}+\sum\limits_{n=1}^{\infty}\frac{H_{n-1}^{2}}{(n-a)(n-b)}
+\frac{\pi \cot(\pi b)}{b-a}\left(\psi(b)+\gamma\right)^{2}
-\frac{\pi \cot(\pi a)}{b-a}\left(\psi(a)+\gamma\right)^{2}\nonumber\\
&\quad-\frac{2}{b-a}
\left\{-\frac{\zeta(3;a)-\zeta(3;b)}{2}+\zeta(2)\left(\psi(b)-\psi(a)\right)\right\}\nonumber\\
&\quad-\frac{2}{b-a}\left\{\sum\limits_{n=1}^{\infty}\left[\frac{H_{n}}{(n+a)^{2}}-\frac{H_{n}}{(n+b)^{2}}\right]
+\zeta(2)\left(\psi(b)-\psi(a)\right)+\sum\limits_{n=0}^{\infty}\frac{H_{n}^{(2)}}{(n+a)(n+b)}(b-a)\right\}\nonumber\\
&=0.
\end{align}
\end{cor}

\begin{cor}  For $a,b\notin \Z$ with $a\neq b$,
\begin{align}
&\sum\limits_{n=1}^{\infty}\frac{[H_{n}^{(2)}]^{2}}{(n+a)(n+b)}+\sum\limits_{n=1}^{\infty}\frac{[H_{n-1}^{(2)}]^2}{(n-a)(n-b)}\nonumber\\
&\quad+2\zeta(2)\sum\limits_{n=1}^{\infty}\frac{H_{n}^{(2)}}{(n+a)(n+b)}-2\zeta(2)\sum\limits_{n=1}^{\infty}\frac{H_{n-1}^{(2)}}{(n-a)(n-b)}\nonumber\\
&\quad+\frac{\pi \cot(\pi b)}{b-a}\left\{\zeta(2;b)^{2}-\zeta(2)^{2}\right\}
-\frac{\pi \cot(\pi a)}{b-a}\left\{\zeta(2;a)^{2}-\zeta(2)^{2}\right\}\nonumber\\
&\quad-\frac{2}{b-a}\left\{-\frac{\zeta(5;a)-\zeta(5;b)}{2}+\zeta(2)\left(\zeta(3;a)-\zeta(3;b)\right)+\zeta(4))\left(\psi(b)-\psi(a)\right)\right\}\nonumber\\
&\quad-\frac{2}{b-a}
\left\{ \begin{array}{l}
-\zeta(2)\left[\zeta(3;a)-\zeta(3;b)\right]-\sum\limits_{n=0}^{\infty}\left[\frac{H_{n}^{(2)}}{(n+a)^{3}}-\frac{H_{n}^{(2)}}{(n+b)^{3}}\right]\\
+2\zeta(3)\left[\zeta(2;a)-\zeta(2;b)\right]-2\sum\limits_{n=0}^{\infty}\left[\frac{H_{n}^{(3)}}{(n+a)^{2}}-\frac{H_{n}^{(3)}}{(n+b)^{2}}\right]\\
-3\zeta(4)\left(\psi(b)-\psi(a)\right)-3(b-a)
\sum\limits_{n=0}^{\infty}\frac{H_{n}^{(4)}}{(n+a)(n+b)}
\end{array} \right\}\nonumber\\
&\quad-\frac{4}{b-a}\zeta(2)\left\{\zeta(2)\left(\psi(b)-\psi(a)\right)+(b-a)\sum\limits_{n=0}^{\infty}\frac{H_{n}^{(2)}}{(n+a)(n+b)}\right\}\nonumber\\
&=0.
\end{align}
\end{cor}

\begin{re} From \cite[Eq. (2.24)]{Xu2017-JMAA} and \cite[Thm. 3.2]{XC2019}, we know that for integer $p\geq 2$, the parametric Euler sums
\begin{align*}
\sum_{n=1}^\infty \frac{H_n}{(n+a)^p} \quad (a\notin \N^-)
\end{align*}
can be expressed in terms of a combination of products of (Hurwitz) zeta function and digamma function. Moreover, if considering the contour integral
\begin{align*}
\oint\limits_{\left( \infty  \right)}\frac{\pi \cot(\pi z)\psi^{(p-1)}(-z)\psi^{(m-1)}(-z)}{(p-1)!(m-1)!}r(z)dz=0,
\end{align*}
by a similar argument as in the proof of Theorem \ref{thm-quadratic-para-ES},
we also obtain a similar evaluation of Theorem \ref{thm-para1-ration-funct-residue}. More general, we can consider the general contour integral
\begin{align*}
\oint\limits_{\left( \infty  \right)}\frac{\pi \cot(\pi z)\psi^{(p_1-1)}(-z)\cdots\psi^{(p_m-1)}(-z)}{(p_1-1)!\cdots(p_m-1)!}r(z)dz=0
\end{align*}
and
\begin{align*}
\oint\limits_{\left( \infty  \right)}\frac{\pi \psi^{(p_1-1)}(-z)\cdots\psi^{(p_m-1)}(-z)}{\sin(\pi z)(p_1-1)!\cdots(p_m-1)!}r(z)dz=0
\end{align*}
to establish some quite general evaluations of (alternating) parametric Euler sums. We leave the detail to the interested reader.
\end{re}

\medskip

{\bf Funding.} The authors is supported by the National Natural Science Foundation of China (Grant No. 12101008), the Natural Science Foundation of Anhui Province (Grant No. 2108085QA01, 2108085QG304) and the University Natural Science Research Project of Anhui Province (Grant No. KJ2020A0057).

\end{document}